\documentclass[10pt,a4paper]{article}
\usepackage[margin=1.1cm]{geometry}
\usepackage{amsmath,amsthm,amsfonts,amssymb,amscd,cite,graphicx}

\usepackage{titlesec}

\usepackage{url}
\titleformat{\section}
{\normalfont\fontsize{12}{15}\bfseries}{\thesection}{1em.}{}

\newtheorem{corollary}{Corollary}[section]
\newtheorem{lemma}{Lemma}[section]

\newtheorem{theorem}{Theorem}[section]

\allowdisplaybreaks[4]

\let\oldbibliography\thebibliography
\renewcommand{\thebibliography}[1]{%
  \oldbibliography{#1}%
  \setlength{\itemsep}{-2pt}%
}

\begin{document}

\baselineskip=0.20in

\makebox[\textwidth]{%
\hglue-15pt
\begin{minipage}{0.6cm}	
\vskip9pt
\end{minipage} \vspace{-\parskip}
\begin{minipage}[t]{6cm}
\footnotesize{ {\bf \phantom{j}} \\ \underline{}}
\end{minipage}
\hfill
\begin{minipage}[t]{6.5cm}
  \normalsize {\it \phantom{a}}  {\bf X} (202X) XX--XX \\
  \normalsize DOI: XXX
\end{minipage}}
\vskip36pt

\hyphenation{Ko-ku-shi-ka-n Fi-gue-ro-a I-chi-shi-ma Mun-ta-ner Ba-tle Cen-te-no}

\noindent
{\large \bf On the strength and domination number of graphs}\\

\noindent
Yukio Takahashi$^{1,}\footnote{Corresponding author (takayu@kokushikan.ac.jp)}$, Rikio Ichishima$^{2}$, Francesc A. Muntaner-Batle$^{3}$\\

\noindent
\footnotesize
$^1${\it Department of Science and Engineering, Faculty of Electronics and
Informatics, Kokushikan University,  4-28-1 Setagaya, Setagaya-ku, Tokyo 154-8515, Japan} \\
$^2${\it Department of Sport and Physical Education, Faculty of Physical Education, Kokushikan University,  7-3-1 Nagayama, Tama-shi, Tokyo 206-8515, Japan} \\
\noindent
$^3${\it Graph Theory and Applications Research Group, School of Electrical Engineering and Computer Science, Faculty of Engineering and Built Environment, The University of Newcastle, NSW 2308, Australia} \\
\noindent

 (\footnotesize Received: Day Month 202X. Received in revised form: Day Month 202X. Accepted: Day Month 202X. Published online: Day Month 202X.)\\

\setcounter{page}{1} \thispagestyle{empty}

\baselineskip=0.20in

\normalsize

 \begin{abstract}
   \noindent
A numbering $f$ of a graph $G$ of order $n$ is a labeling that assigns
distinct elements of the set $\left\{ 1,2,\ldots ,n\right\} $ to the
vertices of $G$. The strength $\textrm{str}_{f}\left( G\right)$ of a numbering
$f:V\left( G\right) \rightarrow \left\{ 1,2,\ldots ,n\right\} $ of $G$ is defined by%
\begin{equation*}
\mathrm{str}_{f}\left( G\right) =\max \left\{ f\left( u\right) +f\left(
v\right) \left| uv\in E\left( G\right) \right. \right\} \text{,}
\end{equation*}%
that is, $\mathrm{str}_{f}\left( G\right) $ is the maximum edge label of $G$
and the strength\ \textrm{str}$\left( G\right) $ of a graph $G$
itself is 
\begin{equation*}
\mathrm{str}\left( G\right) =\min \left\{ \mathrm{str}_{f}\left( G\right)
\left| f\text{ is a numbering of }G\right. \right\} \text{.}
\end{equation*} 
In this paper, we present a sharp lower bound for the strength of a graph in terms of its domination number as well as its (edge) covering and (edge) independence number. 
We also provide a necessary and sufficient condition for the strength of a graph to attain the earlier bound in terms of their subgraph structure. 
In addition, we establish a sharp lower bound for the domination number of a graph under certain conditions. 
 \\[2mm]
 \end{abstract}

\baselineskip=0.20in

\section{Introduction}

We refer to the book by Chartrand and Lesniak \cite{CL} for
graph-theoretical notation and terminology not described in this paper.
In particular, the \emph{vertex set} of a graph $G$ is denoted by $V \left(G\right)$, 
while the \emph{edge set} of $G$ is denoted by $E\left (G\right)$. 

We will use the notation $\left[ a,b\right] $ for the interval of integers $%
x $ such that $a\leq x\leq b$. For a graph $G$ of order $n$, a \emph{%
numbering} $f$ of $G$ is a labeling that assigns distinct elements of the
set $\left[ 1,n\right] $ to the vertices of $G$, where each $uv\in E\left(
G\right) $ is labeled $f\left( u\right) +f\left( v\right) $. The \emph{%
strength} \textrm{str}$_{f}\left( G\right) $ \emph{of a numbering} $%
f:V\left( G\right) \rightarrow \left[ 1,n\right] $ of $G$ is defined by%
\begin{equation*}
\mathrm{str}_{f}\left( G\right) =\max \left\{ f\left( u\right) +f\left(
v\right) \left| uv\in E\left( G\right) \right. \right\} \text{,}
\end{equation*}%
that is, $\mathrm{str}_{f}\left( G\right) $ is the maximum edge label of $G$
and the \emph{strength\ }\textrm{str}$\left( G\right) $ of a graph $G$
itself is 
\begin{equation*}
\mathrm{str}\left( G\right) =\min \left\{ \mathrm{str}_{f}\left( G\right)
\left| f\text{ is a numbering of }G\right. \right\} \text{.}
\end{equation*}%
A numbering $f$ of a graph $G$ for which $\mathrm{str}_{f}\left( G\right) =%
\mathrm{str}\left( G\right) $ is called a \emph{strength labeling} of $G$. 
Since empty graphs $nK_{1}$ do not have edges, this definition does not apply to such
graphs. Consequently, we may define $\mathrm{str}\left( nK_{1}\right)
=+\infty $ for every positive integer $n$. This type of numberings was
introduced in \cite{IMO1} as a generalization of the problem of finding
whether a graph is super edge-magic or not (see \cite{ELNR} for the
definition of a super edge-magic graph, and also consult either \cite{AH} or \cite{FIM} for
alternative and often more useful definitions of the same concept).

There are other related parameters that have been studied in the area of
graph labelings. 
Excellent sources for more information on this topic are
found in the extensive survey by Gallian \cite{Gallian}, which also includes
information on other kinds of graph labeling problems as well as their
applications.

Several bounds for the strength of a graph have been found in terms of
other parameters defined on graphs (see \cite{GLS, IMT, IMO1, IMOT}). Among others, the
following result established in \cite{IMO1} that provides a lower bound for the strength of a graph $G$ in
terms of its order and minimum degree  $\delta\left(G\right)$ is particularly useful.
\begin{lemma}
\label{lemma_trivial}For every graph $G$ of order $n$ with $\delta \left(
G\right) \geq 1$,%
\begin{equation*}
\mathrm{str}\left( G\right) \geq n+\delta \left( G\right) \text{.}
\end{equation*}
\end{lemma}

It is worth to mention that the lower bound given in Lemma \ref%
{lemma_trivial} is sharp in the sense that there are infinitely many graphs $%
G$ for which $\mathrm{str}\left( G\right) =\left| V\left( G\right) \right|
+\delta \left( G\right) $ (see \cite{GLS, IMO1, IMO2, IMOT} for a detailed list
of such graphs and other sharp bounds).

For every graph $G$ of order $n$, it is clear that $3 \leq \mathrm{str}\left( G\right) \leq 2n-1$. 
In fact, it was shown in \cite{IMO4} that for every $k\in \left[ 1,n-1\right] $, 
there exists a graph $G$ of order $n$ satisfying $\delta \left( G\right) =k$ and 
$\mathrm{str}\left( G\right) =n+k$.

In the process of settling the problem (proposed in \cite{IMO1}) of finding
sufficient conditions for a graph $G$ of order $n$ with $\delta \left(
G\right) \geq 1$ to ensure that $\mathrm{str}\left( G\right) =n+\delta
\left( G\right) $, an equivalent definition of the following class of graphs was defined in \cite{IMO3}.
For integers $k\geq 2$, let $F_{k}$ be the graph with $V\left( F_{k}\right)
=\left\{ v_{i}\left\vert i\in \left[ 1,k\right] \right. \right\} $ and 
\begin{equation*}
E\left( F_{k}\right) =\left\{ v_{i}v_{j}\left\vert i\in \left[
1,\left\lfloor k/2\right\rfloor \right] \text{ and }j\in \left[ 1+i,k+1-i%
\right] \right. \right\} \text{.}
\end{equation*}%
Let $\overline{G}$ denote the complement of a graph $G$. The
following result found in \cite{IMO3} provides a necessary and sufficient
condition for a graph $G$ of order $n$ to hold the inequality $\mathrm{str}%
\left( G\right) \leq 2n-k$, where $k\in \left[ 2,n-1\right] $.

\begin{theorem}
\label{main1}Let $G$ be a graph of order $n$. Then $\mathrm{str}\left(
G\right) \leq 2n-k$ if and only if $\overline{G}$ contains $F_{k}$ as a
subgraph, where $k\in \left[ 2,n-1\right] $.
\end{theorem}

The preceding result plays an important role in the study of the strength of graphs (see \cite{IMOT2, IOT, IOT2} for instance).
The following result was deduced from Lemma \ref{lemma_trivial} and Theorem %
\ref{main1}.

\begin{theorem}
\label{main3}Let $G$ be a graph of order $n$ with $\delta \left( G\right)
=n-k$, where $k\in \left[ 2,n-1\right] $. Then $\mathrm{str}\left( G\right)
=n+\delta \left( G\right)$ if and only if $\overline{G}$ contains $F_{k}$ as a subgraph.
\end{theorem}

\section{Results involving domination number}

In this section, we present some results involving the domination number of a graph and a new sharp lower bound for the strength of a graph without isolated vertices.

The following result provides a lower bound for the strength of a graph in terms of its domination number.

\begin{lemma}
\label{strength-domination}For every graph $G$ of order $n$,
\begin{equation*}
\mathrm{str}\left( G\right) \geq 2n-2\gamma\left( G\right)+1\text{.}
\end{equation*}
\end{lemma}

\begin{proof}
Let $G$ be a graph with $V\left(G\right)=\left\{ v_{i}\left\vert i\in \left[ 1,n\right] \right. \right\} $, and consider a strength labeling $f$ of $G$. 
Since $1 \leq \gamma\left(G\right) \leq n$, it follows that the set 
\begin{equation*}
S=\left[ n-\gamma\left(G\right),n\right]
\end{equation*}
contains at least two integers.
By the pigeonhole principle, at least two integers in $S$ are assigned to two adjacent vertices, say $f\left(v_{s}\right)$ and $f\left(v_{t}\right)$, 
where $s,t\in \left[ 1,n\right]$. Now, assume, without loss of generality, that $f\left(v_{s}\right)>f\left(v_{t}\right)$. Then 
\begin{eqnarray*}
\min \left\{ f\left(v_{s}\right)+f\left(v_{t}\right)\left| s,t\in \left[ 1,n\right] \right. \right\}
&\geq& \left(n-\gamma\left( G\right)\right)+\left(n-\gamma\left( G\right)+1\right) \\
&=&2n-2\gamma \left( G\right)+1 \text{.}
\end{eqnarray*}%
Thus,
\begin{eqnarray*}
\mathrm{str}\left( G\right) = \mathrm{str}_{f}\left( G\right) \geq n-2\gamma \left( G\right)+1 \text{,}
\end{eqnarray*}
completing the proof.
\end{proof}

The bound given in Lemma \ref{strength-domination} is sharp in the sense that there are infinitely many graphs $G$ for which $\mathrm{str}\left( G\right) =2\left| V\left( G\right) \right|-2\gamma \left( G\right) +1$. To see this, it suffices to consider the complete graph $K_{n}$ of order $n$. 
It is straightforward to see that $ \mathrm{str}\left( K_{n}\right)=2n-1$ and $\gamma\left(K_{n}\right)=1$ ($n\geq 2$). 
This implies that $\mathrm{str}\left( K_{n}\right) =2n-2\gamma \left( K_{n}\right)+1$ ($n\geq 2$).

The following result provides a necessary and sufficient
condition for a graph $G$ of order $n$ to hold for $\mathrm{str}\left(G\right) =2n-2\gamma \left( G\right)+1$.

\begin{theorem}
\label{main_result}Let $G$ be a graph of order $n$ with $\gamma \left( G\right)=k$, where $k\in \left[ 2,\lceil n/2 \rceil \right]$. 
Then $\mathrm{str}\left(G\right) =2n-2\gamma \left( G\right)+1$ if and only if $\overline{G}$ contains $F_{2k-1}$ as a subgraph.
\end{theorem}

\begin{proof}
First, suppose that $\mathrm{str}\left( G\right) =2n-2k+1$, where $\gamma \left( G\right)=k$ ($k\in \left[ 2,\lceil n/2 \rceil \right]$). Let $V\left(
G\right) =\left\{ v_{i}\left| i\in \left[ 1,n\right] \right. \right\} $, and
assume, without loss of generality, that there exists a strength labeling of 
$G$ that assigns $i$ to $v_{i}$ ($i\in \left[ 1,n\right] $). 
Since $\mathrm{str}\left( G\right) = 2n-2k+1$, every two vertices $v_{i}$ and $%
v_{j}$ for which $i+j>2n-2k+1$ are not adjacent in $G$. 
This means that every two vertices $v_{i}$ and $v_{j}$ for which $i+j>2n-2k+1$ are adjacent in  $\overline{G}$. 
Let $v_{i}=w_{n+1-i}$ ($i \in  \left[ 1,n\right] $) so that   $V\left(\overline{G}\right)= \left\{ w_{i}\left| i\in \left[ 1,n\right] \right.\right\}$.
Then if $w_{n+1-i}$ and $w_{n+1-j}$ are adjacent in $\overline{G}$, it follows that 
\begin{eqnarray*}
\left( n+1-i\right) +\left( n+1-j\right)&=&2n+2-\left( i+j\right) \\
&<&2n+2-\left( 2n-2k+1\right) =2k+1\text{.}
\end{eqnarray*}%
Thus, $\overline{G}$ contains $F_{2k-1}$ as a subgraph.

Next, suppose that $\overline{G}$ contains $F_{2k-1}$ as a subgraph, where $\gamma \left( G\right)=k$ ($k\in \left[ 2,\lceil n/2 \rceil \right]$). 
It follows from Theorem \ref{main1} that 
\begin{eqnarray*}
\mathrm{str}\left( G\right) \leq 2n-\left(2k-1\right)=2n-2\gamma \left( G\right)+1 \text{.}
\end{eqnarray*}
\end{proof}

The following result found in \cite{IMT} provides a necessary and sufficient
condition for a graph $G$ of order $n$ to hold for $\mathrm{str}\left(G\right) =2n-2\beta \left( G\right)+1$, 
where $\beta \left( G\right)$ denotes the independence number of $G$.

\begin{theorem}
\label{strength-independence-number}Let $G$ be a graph of order $n$ with $\beta\left( G\right)=k$, where $k\in \left[ 2,\lceil n/2 \rceil \right]$. 
Then $\mathrm{str}\left(G\right) =2n-2\beta \left( G\right)+1$ if and only if $\overline{G}$ contains $F_{2k-1}$ as a subgraph.
\end{theorem}

The next result follows from Theorems \ref {main_result} and \ref {strength-independence-number}, which shows connection between the domination number and independence number.

\begin{corollary}
\label{connection}Let $G$ be a graph of order $n$ with $\gamma \left( G\right)=k$, where $k\in \left[ 2,\lceil n/2 \rceil \right]$,
and assume that $\overline{G}$ contains $F_{2k-1}$ as a subgraph.
Then $\gamma\left( G\right)=\beta \left( G\right)$.
\end{corollary}

The next lower bound for the domination in terms of its minimum degree is obtained immediately from Lemma \ref{strength-domination}.

\begin{corollary}
\label{corollary-domination}
Let $G$ be a graph of order $n$ with $\mathrm{str}\left( G\right)=n+\delta\left(G\right)$, where $\delta\left(G\right) \geq 1$. Then
\begin{equation*}
\gamma \left( G\right)\geq \lceil\left(n-\delta\left(G\right)+1\right) /2\rceil \text{.}
\end{equation*}
\end{corollary}

There are infinitely many graphs attaining the bound given in Corollary \ref{upper_bound_domination_number}. For instance, if $G=K_{n}$ ($n\geq 2$), 
then $\mathrm{str}\left(G\right)=2n-1$. Also, we have 
\begin{equation*}
\left|V\left(G\right)\right|=n \text{ and } \delta\left(G\right)=n-1 \text{.} 
\end{equation*}
This implies that $\left|V\left(G\right)\right|-\delta\left(G\right)=1$. On the other hand, we have 
\begin{equation*}
\gamma\left(G\right)=1 \text{ and } \lceil\left(\left|V\left(G\right)\right|-\delta\left(G\right)+1\right)/2\rceil=1 \text{.}
\end{equation*}

It is known that the domination number of a graph without isolated vertices is bounded above by all of the covering and independence numbers (see \cite[p. 307]{CL} for instance). Note that we denote $\alpha\left( G\right)$, $\alpha_{1}\left( G\right)$ and $\beta_{1}\left( G\right)$ to be the covering, edge covering and edge independence numbers of $G$, respectively. 

\begin{theorem}
\label{upper_bound_domination_number}If $G$ is a graph without isolated vertices, then
\begin{equation*}
\gamma\left(G\right) \leq \min \left\{\alpha\left( G\right), \alpha_{1}\left( G\right),\beta\left( G\right), \beta_{1}\left( G\right)\right\} \text{.}
\end{equation*}
\end{theorem}

The preceding theorem together with Lemma \ref{strength-domination} gives the following lower bound for the strength of a graph without isolated vertices.

\begin{corollary}
\label{corollary-covering-independence-number}For every graph $G$ without isolated vertices,
\begin{equation*}
\mathrm{str}\left( G\right) \geq 2n-2\min \left\{\alpha\left( G\right), \alpha_{1}\left( G\right), \beta\left( G\right), \beta_{1}\left( G\right)\right\}+1 \text{.}
\end{equation*}
\end{corollary}

It is known from \cite {IMO1}  that $\mathrm{str}\left( C_{2n+1}\right) =2n+3$ ($n\geq 1$). Also, note that 
\begin{equation*}
\alpha\left(C_{2n+1}\right)=\alpha_{1}\left(C_{2n+1}\right)=n+1 
\text{ and } 
\beta\left(C_{2n+1}\right)=\beta_{1}\left(C_{2n+1}\right)=n \left(n\geq 1\right) \text{.} 
\end{equation*}
By means of these, it indicates that the bound given in Corollary \ref{corollary-covering-independence-number} is sharp.


\end{document}